\documentclass[12pt]{article}
\usepackage[margin=1.25in]{geometry}
\usepackage[english]{babel}
\usepackage[nottoc]{tocbibind} % Includes "References" in the table of contents
\usepackage{etoolbox}
\usepackage{textcomp}
\usepackage{amsmath}
\usepackage{amsthm}
\usepackage{thmtools}
\usepackage{amssymb}
\usepackage{bbm}
\usepackage{graphicx}
\usepackage{tikz}
\usepackage{mathtools}
\usepackage[shortlabels]{enumitem}
\usepackage{hyperref}
\usetikzlibrary{cd}
\usepackage{cleveref}
\makeatletter

\def\env@sqcases{%
  \let\@ifnextchar\new@ifnextchar
  \left\lbrack
  \def\arraystretch{1.2}%
  \array{@{}l@{\quad}l@{}}%
}
\makeatother

\hypersetup{
    colorlinks=true,
    linkcolor=blue,
    filecolor=magenta,
    pdftitle={Zhizhin 2nd Year Paper},
    }
    
\newcommand{\R}{\ensuremath{\mathbb R}}
\newcommand{\C}{\ensuremath{\mathbb C}}
\newcommand{\Z}{\ensuremath{\mathbb Z}}

\newcommand{\Q}{\ensuremath{\mathbb Q}}

\newcommand{\A}{\ensuremath{\mathbb A}}
\newcommand{\E}{\ensuremath{\mathcal E}}
\newcommand{\defeq}{\vcentcolon=}
\renewcommand{\P}{\ensuremath{\mathcal P}}
\renewcommand{\O}{\ensuremath{\mathcal{O}}}

\DeclareMathOperator{\rk}{rk}

\DeclareMathOperator{\Spec}{Spec}
\DeclareMathOperator{\id}{id}

\DeclareMathOperator{\codim}{codim}
\DeclareMathOperator{\Supp}{Supp}

\DeclareMathOperator{\Vol}{Vol}
\DeclareMathOperator{\MVol}{MVol}
\DeclareMathOperator{\Ker}{Ker}
\DeclareMathOperator{\Hom}{Hom}
\DeclareMathOperator{\Conv}{Conv}
\DeclareMathOperator{\GL}{GL}
\renewcommand{\char}{\ensuremath{\operatorname{char}}}

\newcommand{\citestacks}[1]{\cite[\href{https://stacks.math.columbia.edu/tag/#1}{Tag #1}]{stacks-project}}
            
\theoremstyle{definition}
\newtheorem{theorem}{\textbf{Theorem}}[section]
\newtheorem{example}[theorem]{\textbf{\sc{Example}}}
\newtheorem*{example*}{\textbf{\sc{Example}}}
\newtheorem{claim}[theorem]{\textbf{\sc{Claim}}}
\crefname{claim}{Claim}{Claims}
\newtheorem*{claim*}{\textbf{\sc{Claim}}}
\crefname{claim*}{Claim}{Claims}
\newtheorem{notation}[theorem]{\textbf{\sc{Notation}}}
\crefname{notation}{Notation}{Notations}
\newtheorem{lemma}[theorem]{\textbf{\sc{Lemma}}}
\crefname{lemma}{Lemma}{Lemmas}
\newtheorem{remark}[theorem]{\textbf{\sc{Remark}}}
\newtheorem{corollary}[theorem]{\textbf{\sc{Corollary}}}
\crefname{corollary}{Corollary}{Corollaries}

\theoremstyle{plain}
\newtheorem{definition}[theorem]{Definition}

\usepackage{sectsty}
\sectionfont{\centering\sc}
\makeatletter
\def\@seccntformat#1{\csname the#1\endcsname.\ }
\makeatother

\makeatletter
\let\@fnsymbol\@arabic
\makeatother

\addto\extrasenglish{

}

\title{
  \sc{Irreducibility of Toric Complete Intersections}
  }
\author{
  \sc{Andrey Zhizhin\thanks{National Research University Higher School of Economics, Russian Federation}}
  }
\date{July 2024\\
Moscow}
\begin{document}
\maketitle

\begin{abstract}    
    We develop an approach to study the irreducibility of generic complete intersections in the algebraic torus defined by equations with fixed monomials and fixed linear relations on coefficients. Using our approach we generalize the irreducibility theorems of Khovanskii from \cite{Khovanskii2016} to fields of arbitrary characteristic. Also we get a combinatorial sufficient conditions for irreducibility of engineered complete intersections (notion introduced in \cite{Esterov2024}). As an application we give a combinatorial condition of irreducibility for some critical loci and Thom-Bordmann strata:  $f = f'_x = 0$, $f'_x = f'_y = 0$, $f = f'_x = f'_{xx} = 0$, where f is a generic Laurent polynomial with a prescribed monomial set.
\end{abstract}
\pagebreak

\begin{center}
    {\bfseries Acknowledgements} 
\end{center}
I would like to thank my advisor Alexander Esterov for all the gentle guidance and peristent discussions that formed my whole understanding of the Newton Polytope Theory. I am also grateful to Vlad Safronov for reviewing and discussing some parts of my work which in particular resulted in a much more general formulation of \autoref{cor:ECI_simpler_poly}.
The study has been funded within the framework of the HSE University Basic Research Program.

\tableofcontents

%sections
\section{Introduction}
% change subsection title
\subsection{To make long story short}
Let $k$ be an arbitrary field, $T^n \defeq \Spec k\left[ x_1^{\pm1}, \dots, x_n^{\pm 1} \right]$ be the algebraic $n$-dimensional torus over $k$, $f \in \Gamma(T^n, \O) =  k\left[ x_1^{\pm1}, \dots, x_n^{\pm 1} \right] $ be a Laurent polynomial in $n$ variables. Then we call the finite set
\[
\Supp f \defeq \left\{ x_1^{d_1} \cdot \hdots \cdot x_n^{d_n} \middle|\ c_{\mathbf d} \ne 0 \right\}, \text{ where } f = \sum_{\mathbf d \in \Z^n} c_{\mathbf d} x_1^{d_1} \cdot \hdots \cdot x_n^{d_n},
\]
{\itshape \bfseries the support set of the Laurent polynomial $f$}. The study of complete intersections in $T^n$ in terms of the support sets of the equations is usually called the Newton Polytope Theory. In this paper we address the question of irreducibility in this setting.

Denote by $M \simeq \Z^n$ the monomial lattice. For a finite subset $A \subset M$ consider $k^A \defeq \left\{ f \in \Gamma(T^n, \O)\ |\ \Supp f \subset A \right\}$ --- the space of polynomials supported at $A$. Then $k^{A_1} \times \cdots \times k^{A_m}$ could be viewed as the space of systems of equations such that the $i$-th equation is supported at $A_i$. Classical theorems by Khovanskii (cf. \cite{Khovanskii2016}) give a criterion for irreducibility of a variety defined by the general system of equations from $\C^{A_1} \times \dots \times \C^{A_m}$. In this text we call that criterion\footnote{technically the Khovnaskii condition is not a criterion of irreducibility, but rather a sufficient condition. However, Khovanskii in his work gives a precise criterion and we generalize his criterion to arbtrary charasteristic} the Khovanskii condition, cf. \autoref{def:khovanskii_condition}. We give a new proof for these results as well as generalize them to arbitrary characteristic in \autoref{sec:khovanskii}. 

However the main point of this work is to extend the Newton polytope theory beyond general systems of equations. Indeed, in the many topics where this theory is classically applied, one can see growing role of slightly degenerate systems of equations, in the following sense: some equations in the system are, as before, general linear combinations of given collections of monomials, but the others are obtained from them by e.g. permuting variables or taking partial derivatives. Systems of equations that are general modulo a symmetry occur e.g. in the study of nearly rational varieties (starting at least from examples in \cite{Bayle1994}) or in Galois theory (see \cite{Esterov2022}), which, further, surveys similar examples from several other topics, and gives a version of Khovanskii's irreducibility theorem for such general symmetric systems of equations).
We focus on another extension of the classical Newton polytope setting: general systems, in which some equations are partial derivatives of the others. This includes e.g. examples from enumerative geometry (see \cite{Esterov2013} and references therein) and polynomial optimization (see e.g. \cite{DHOBT}, \cite{BSW}, \cite{LNRW}). For instance, given a generic hypersurface defined by $f = 0$ with a prescribed Newton polytope, its coordinate projection has the critical locus $f = f'_x = 0$ and the higher Thom--Bordmann strata $f = f'_x = f'_{xx} = \dots =0$, which are important for enumerative singularity theory. This naturally extends further to so-called engineered complete intersections \cite{Esterov2024}, in which the coefficients of the equations are general modulo given linear dependencies among them. This generality includes many new interesting objects, such as hyperplane arrangement complements, generalized complete intersection Calabi--Yau varieties (studied since \cite{AAGGL}), and some other examples originating from mathematical physics (e.g. \cite{BMMT}). We extend Khovanskii's irreducibility theorem to such objects.

\paragraph{Results} The main goal we pursue is to study a new class of toric complete intersections, viz. the varieties defined by general systems from vector subspaces of $k^{A_1} \times \cdots \times k^{A_m}$, i.e. we allow fixed liner relations on coefficients. In the said generality we are able to give a rather abstract  condition, \autoref{thm:sufficient}. In the specific context of Engineered Complete Intersections\footnote{see \autoref{sssec:prelim_engineered} and \autoref{sec:ECI} or the original preprint \cite{Esterov2024}} we have an explicit combinatorial condition, \autoref{thm:ECI}. In particular, as a concrete application we get\footnote{\autoref{cor:ECI_simpler_poly}}:
\begin{claim*}
    Let $A$ be a finite set of Laurent monomials. Fix any function $l: A \to k$. Let $p_1, \dots, p_m \in k[T]$ be polynomials s.t. $\deg p_i = i - 1$. For $f = \sum_{\chi \in A} c_\chi \cdot \chi \in k^A$ define $g_i(f) \defeq \sum_\chi p_i(l(\chi)) c_\chi \cdot \chi$. If there are at least $m$ fibres of the function $l$ each of dimension\footnote{see \autoref{sssec:techincal_codim} for the formal definition of dimension} at least $m + 1$, then for the general $f \in k^A$ the system $g_1(f) = ~\cdots = ~g_m(f) = 0$ defines an irreducible variety.
\end{claim*}
A direct corollary of the above claim is the following example. For simplicity until the end of this paragraph we assume that $\char k = 0$. In \autoref{ssec:ECI_critical} one can find the same examples in full generality, in particular, in arbitrary characteristic.
\begin{example*}
    Fix a coordinate system $x, y_1, \dots, y_{n - 1}$ in $T^n$ and a finite set of monomials $A \subset M$. Define 
    $H_d \defeq \{ \chi \in M\ |\ \deg_x \chi = d\}$. 
    If there are distinct $d_1, \dots, d_{m + 1} \in \Z$ such that $\dim A \cap H_{d_i} > m + 1$ for all $i$, then the locus
    \[
    f = \frac{\partial}{\partial x}f = \cdots = \frac{\partial^m}{\partial x^m} f = 0
    \]
    is irreducible for the general $f \in k^A$.
\end{example*}

Another example that does not follow from the above claim, but is a corollary of our general \autoref{thm:ECI}:
\begin{example*}
    Fix a coordinate system $x, y, z_1 \dots, z_{n - 2}$ in $T^n$ and a finite set $A \subset M$. Consider the map $c: A \to \Q^2$, $\chi \mapsto (\deg_x \chi, \deg_y \chi)$. Assume that there is a line\footnote{we assume that $0 \in l$} $l \subset \Q^2$ such that both $A \backslash c^{-1}(l)$ and $c^{-1} (l \backslash 0)$ are of dimension at least 3. Then for the general $f \in k^A$ the following locus is irreducible:
    \[
    \frac{\partial}{\partial x} f = \frac{\partial}{\partial y} f = 0.
    \]
\end{example*}

%\subsection{Results}
%We have two main results. The first one --- rather abstract-theoretical --- \autoref{thm:sufficient}. It is a sufficient condition for irreducibility of the variety defined by a general system with fixed monomials and fixed linear relations on coefficients. It is not that easy to use that condition on actual examples, but it has the advantage of being computable meaning that by using computer algebra one can check the condition on any given set of monomials and linear relations.

%From our theoretical result we derive a more applicable condition --- \autoref{thm:ECI}. It provides an explicit combinatorial condition on irreducibility for engineered complete intersections --- notion coined in \cite{Esterov2024} that covers many classes of non-classical systems. In particular, we re-prove and generalize to arbitrary characteristic the irreducibility results of \cite{Khovanskii2016} --- see
%\autoref{thm:khovanskii_irreducibility}, \autoref{thm:khovanskii_components} --- we prove them separately, but \autoref{thm:khovanskii_irreducibility} can be considered as a special case of \autoref{thm:ECI}.

\subsection{Paper Structure}
In \autoref{sec:preliminaries} we set up the notation, recall some basic facts from the Newton Polytope Theory (\autoref{ssec:prelim_newton}) and give a quick introduction to engineered complete intersections (\autoref{sssec:prelim_engineered}) which is later used in \autoref{sec:ECI}; also in \autoref{ssec:prelim_technical} we define the geometric irreducibility and prove all the general technical statements we need.

We formulate and prove the sufficient condition for irreducibility in its most general form in \autoref{sec:general} --- it is \autoref{thm:sufficient}.
%and further in this paper we use only this condition. There is also another condition which is more complicated than \autoref{thm:sufficient} but has the advantage of being a full criterion --- it is \autoref{thm:criterion}.

Then in \autoref{sec:khovanskii} as the first application of our method we generalize to arbitrary characteristic the results on irreducibility from the classical setting --- namely, the theorems of Khovanskii counting the number of irreducible component, \cite{Khovanskii2016} --- see \autoref{thm:khovanskii_irreducibility} and \autoref{thm:khovanskii_components}.

Finally, in \autoref{sec:ECI} we give a combinatorial sufficient condition for irreducibility for all engineered complete intersection: \autoref{thm:ECI}. As an application, we provide examples on how this condition works for some critical loci: $f = f'_x = 0$, $f'_x = f'_y = 0$, $f = f'_x = f'_{xx} = 0$, etc. 
\section{Preliminaries}\label{sec:preliminaries}
\subsection{Newton Polytope Theory}\label{ssec:prelim_newton}
Here we do some groundwork that is necessary whenever we study complete intersections in the torus in terms of their monomials.
Let us note that the lattice of monomials can be defined in a coordinate-free manner: $M \defeq \Hom(T^n, T^1) \simeq \Z^n$ --- it is a canonical unordered basis in $\Gamma(T^n, \mathcal O_{T^n})$, so for $f \in \Gamma(T^n, \O_{T^n})$ we define $\Supp f \subset M$ as the set of characters such that the coordinates of $f$ with respect to these characters are non-zero. Until the end of this section we fix non-empty finite subsets $A_1, \dots, A_m \subset M$. Then we have the spaces $k^{A_1}, \dots, k^{A_m}$, where $k^{A_i} \defeq \{f \in \Gamma(T^n, \mathcal O_{T^n})\ |\ \Supp f \subset A_i \}$. We denote by $k^{A_\bullet} \defeq k^{A_1} \times \dots \times k^{A_m}$ the space of all systems. The specific class of systems we are interested in should be denoted by $\P \subset k^{A_\bullet}$ --- it must be a vector subspace\footnote{i.e. this class must be defined by linear relations on coefficients of the polynomials}.
\subsubsection{Notation}

\begin{definition}\label{def:evaluation}
    Consider the evaluation morphism $k^{A_i} \times T^n \to \A^1_k \times T^n$, $(f, p) \mapsto~(f(p), p)$ --- it is a morphism of vector bundles\footnote{i.e. it is fiberwise linear}. Then we have the morphism of vector bundles $\E: k^{A_\bullet} \times T^n \to \A_k^m \times T^n$. We will denote by $X$ the kernel of $\E$, i.e.
    \[
    X \defeq \Ker \mathcal E =  \left\{ (\mathbf f, p) \in k^{A_\bullet} \times T^n : \mathbf f(p) = 0 \right\}.
    \]
    By $X_\P$ we denote the base change of $X \to k^{A_\bullet}$ with respect to $\P \to k^{A_\bullet}$, in other words:
    \[
    X_\P = \{ (\mathbf f, p) \in \P \times T^n\ |\ \mathbf f(p) = 0 \}
    \]
\end{definition}

\begin{notation}
    Consider the embedding $\iota : \P \to k^{A_\bullet}$. Then $\E|_\P \defeq \E \circ (\iota, \id_{T^n})$ is a vector bundle morphism $\P \times T^n \to \A^m_k \times T^n$ and clearly $X_\P = \Ker \E|_\P$. Note that if $\rk \E|_\P$ is constant on $T^n$, then $X_\P$ is a vector bundle over $T^n$ of rank $\dim \P - \rk \E|_\P$.
\end{notation}

\begin{remark}
    The space $k^{A_\bullet}$ comes with almost canonical\footnote{'almost' means that the set of coordinate functions is canonical but we have to choose an order} coordinates --- the coefficients of the polynomials. Assume that $A_i = \{ \chi^i_1, \dots, \chi^i_{r_i} \},\ 1 \le i \le m$. Then the matrix of the linear map $\mathcal E(p), p \in T^n$ with respect to the those coordinates is:
    \[
    \begin{pmatrix}
        \chi_1^1(p) & \cdots & \chi_{r_1}^1(p) & 0 & \cdots & 0 & \cdots & 0 & \cdots & 0 \\
        0 & \cdots & 0 & \chi^2_1(p) & \cdots & \chi^2_{r_2}(p) & \cdots & 0 & \cdots & 0 \\
        \vdots & \ddots & \vdots & \vdots & \ddots & \vdots & \ddots & \vdots & \ddots & \vdots \\
        0 & \cdots & 0 & 0 & \cdots & 0 & \cdots & \chi^m_1(p) & \cdots & \chi^m_{r_m}(p)
    \end{pmatrix}
    \]
\end{remark}

\begin{claim}
    $\rk \E \equiv m$ on $T^n$. 
\end{claim}
\begin{proof}
    For all $i, j$ we have $\chi_j^i (p) \ne 0$ on $T^n$, so the rank is always equal to $m$.
\end{proof}

\begin{corollary}
    $X$ is a vector bundle of rank $n - m$ over $T^n$. In particular, $X$ is irreducible and $\dim X = \dim k^{A_\bullet} + n - m$.
\end{corollary}

\subsubsection{Engineered Complete Intersections}\label{sssec:prelim_engineered}
Engineered Complete Intersections (ECI) is a non-classical setting for Newton Polytope Theory proposed by Alexander Esterov in \cite{Esterov2024}. Here we give all the necessary definitions to work with ECI. For a complete account of the notion see the original preprint \cite{Esterov2024}.
\begin{definition}
    We define the inner product $\ast : \Gamma(T^n, \mathcal O) \times \Gamma(T^n, \mathcal O) \to \Gamma(T^n, \mathcal O)$ as follows. For $\chi, \mu \in M$:
    \[
    \chi * \mu = 
    \begin{cases}
        \chi, \text{ if } \chi = \mu \\
        0, \text{ otherwise}
    \end{cases}
    \]
    Then we extend $*$ from $M \times M$ to $\Gamma(T^n, \mathcal O) \times \Gamma(T^n, \mathcal O)$ by $k$-bilinearity.
\end{definition}
\begin{remark}
    For any two $f, g \in k^A$ we have $f * g \in k^A$.
\end{remark}
\begin{definition}
    Let $A \subset M$ be a finite subset and $c_1, \dots, c_d \in k^A$ be linearly independent polynomials. Then we have the linear map $k^A \to (k^A)^d, f \mapsto (c_1 * f, \dots, c_d * f)$. The variety defined by $c_1 * f = \dots = c_d * f = 0$ for $f \in k^A$ s.t. $\Supp f = A$ is called an \textbf{\textit{engineered complete intersection}}.
    
    Let $S_1, \dots, S_m$ be engineered complete intersections. Then $\bigcap g_i S_i$ for general $g_i \in ~T^n(k)$, $T^n(k) \simeq ~(k^\times)^n$ is called an \textit{\textbf{$m$-engineered complete intersection}}.
\end{definition}

\begin{remark}
    In this paper we will study general ECI, i.e. systems $c_1 * f = \dots = c_d * f = 0$ for general $f \in k^A$.
\end{remark}

\begin{remark}\label{rem:prelim_engin_m_ECI_is_ECI}
    \textbf{Every $m$-ECI is a $1$-ECI}: fix $m$ ECI $S_1, \dots, S_m$ such that $S_i$ are defined by $c^i_1 * f_i = \cdots = c^i_{d_i} * f_i = 0$ for $c^i_j, f_i \in k^{A_i}$, where $A_i = \Supp f_i$. Let $\chi_1, \dots, \chi_m \in M$ be such that $(\chi_i \cdot A_i) \cap (\chi_j \cdot A_j) = \varnothing$ for $i \ne j$. Note that the equations $(\chi \cdot c^i_j) * (\chi \cdot f_i) = 0$ and $c^i_j * f = 0$ are equivalent $\forall \chi \in M$, so without loss of generality we could assume that $A_i \cap A_j = \varnothing$ for all $i \ne j$.
    
    Now, define $A \defeq A_1 \sqcup \cdots \sqcup A_m$. Naturally, $k^{A_i} \subset k^A$, so $c^i_j \in k^A$. Put $f = ~f_1 + ~\cdots + ~f_m$. Consider the 1-ECI 
    \[
    S \defeq \{ p \in T^n\ |\ c^i_j * f = 0 \text{ for all } 1 \le i \le m,\ 1 \le j \le d_i \}.
    \]
    Then $S$ coincides with the $m$-ECI $S_1 \cap \cdots \cap S_m$.
\end{remark}

\begin{notation}
    Fix finite sets $A_1, \dots, A_m \subset M$ and $c^i_1, \dots, c^i_{d_i} \in A_i$ that are linearly independent for all fixed $i$. That gives us linear morphisms
    \begin{gather*}
        \mathbf c^i : k^{A_i} \to \left( k^{A_i} \right)^{d_i}, f \mapsto (c^i_1 * f, \dots, c^i_{d_i} * f) \\
        \mathbf c = \oplus_{i = 1}^l \mathbf c^i: k^{A_\bullet} \to \bigoplus_{i = 1}^m (k^{A_i})^{d_i}    
    \end{gather*}
    We denote the image of $\mathbf c$ by $\P = \P(\mathbf c)$. 
\end{notation}

\begin{remark}
    Consider the case $m = 1$, i.e. just the engineered complete intersection $c_1 * f = \dots = c_d * f = 0$ for the general $f \in k^A$ and fixed $c_i \in A = \{ \chi_1, \dots, \chi_r \}$. Then $\mathcal E|_\P$ is the vector bundle morphism 
    \[
    k^A \times T^n \to \A^d \times T^n, (f, p) \mapsto ((c_1 * f)(p), \dots, (c_d * f)(p), p).
    \]
    In particular, the matrix of $\mathcal E|_\P$ over $p \in T^n$ is:
    \[
    \begin{pmatrix}
        (c_1 * \chi_1) (p) & (c_1 * \chi_2) (p) & \cdots & (c_1 * \chi_r) (p) \\
        (c_2 * \chi_1) (p) & (c_2 * \chi_2) (p) & \cdots & (c_2 * \chi_r) (p) \\
        \vdots & \vdots & \ddots & \vdots \\
        (c_d * \chi_1) (p) & (c_d * \chi_2) (p) & \cdots & (c_d * \chi_r) (p)
    \end{pmatrix}
    \]
\end{remark}

\begin{claim}\label{claim:prelim_engin}
    Consider the evaluation morphism $\E : \bigoplus_{i = 1}^l (k^{A_i})^{d_i} \to \A^{d_1 + \hdots + d_m}$ as in the previous subsection. Then $\rk \E|_\P \equiv d_1 + \dots + d_m$ on $T^n$.
\end{claim}
\begin{proof}
Put $\P_i \defeq \mathbf c^i (k^{A_i})$. Then $\P = \P_1 \oplus \cdots \oplus \P_m$. Clearly $\E|_\P = \E|_{\P_1} \oplus \cdots \oplus \E|_{\P_m}$, so without loss of generality we assume $m = 1$ and $\P = \P_1$. Denote by $C$ the matrix of $\mathbf c$: it has $d = d_1$ rows and its columns are indexed by $\chi \in A = A_1 = \{ \chi_1, \dots, \chi_r \}$. Then over $p \in T^n$ we have that $\E|_\P = C \cdot \operatorname{diag}(\chi_1(p), \dots, \chi_r(p)).$ Since all $c_j$ are linearly independent, we know that the rows of $C$ are linearly independent, i.e. $\rk C = d$. So, $\rk \E|_\P = \rk C = d$ as multiplying by an invertible matrix does not affect the rank.
\end{proof}
\begin{corollary}
    Put $X_\P \defeq \Ker \E|_\P$. Then $X_\P$ is a vector bundle over $T^n$. In particular, $X_\P$ is irreducible and $\dim X_\P = n + \sum_i (|A_i| - d_i)$.
\end{corollary}

\begin{remark}
    Every $m$-engineered complete intersection is a fibre of the projection $X_\P \to \P \cong k^{A_\bullet}$ and for the general $\mathbf f \in k^{A_\bullet}$ the fibre $(X_\P)_{\mathbf f}$ is an $m$-engineered complete intersection.
\end{remark}

\subsubsection{Kouchnirenko-Bernstein Formula}
Here we recall a classical result that laid the foundations of Newton Polytope theory.
\begin{definition}
    For two subsets $A, B \subset \R^n$ we define $A + B \defeq \{ a + b\ |\ a \in A, b \in B\}$ --- \textbf{the Minkowski sum}.
\end{definition}

\begin{definition}
    Let $L$ be a lattice, i.e. $L \simeq \Z^n$. We define the \textbf{lattice volume} with respect to $L$ as the unique Euclidean volume form $\Vol_L$ on $L_\R$ such that $\Vol_L (\Delta) = 1$, where\footnote{by $\Conv$ we denote the convex hull} $\Delta = \Conv \{ 0, e_1, \dots, e_n \}$ and $e_1, \dots, e_n$ is a basis of $L$.
\end{definition}

\begin{remark}
    For any finite subset $S \subset L$ we have that $\Vol_L(\Conv S)$ is an integer because $\Conv S$ admits a triangulation by simplicies with vertices in $L$.
\end{remark}

\begin{remark}
    Recall that a polytope is the convex hull of finitely many points. One can easily see that $\Conv (A + B) = \Conv A + \Conv B$, so the sum of any two polytopes is a polytope. It means that given a real space $V$ the set of all polytopes $\operatorname{Pol}(V)$ from $V$ is naturally a monoid with the operation of Minkowski sum and $\{ 0 \}$ as the neutral element.
\end{remark}

\begin{definition}
    Let $L$ be a lattice of rank $n$. The \textbf{lattice mixed volume} with respect to $L$ is the unique function $\MVol_L: \operatorname{Pol}(L_\R)^n \to \R_+$ that satisfies:
    \begin{itemize}
        \item \textit{Linearity:} $\MVol_L(P_1 + P', P_2, \dots P_n) = \MVol_L (P_1, P_2, \dots, P_n) + \MVol_L(P', \dots, P_n)$ for all $P', P_i \in \operatorname{Pol}(L_\R)$;
        
        \item \textit{Symmetricity:} $\MVol_L(P_1, \dots, P_n) = \MVol_L(P_{\sigma(1)}, \dots, P_{\sigma(n)})$ for all $\sigma \in S_n$ and $P_i \in\operatorname{Pol}(L_\R)$;

        \item \textit{Diagonal volume:} $\MVol_L(P, \dots, P) = \Vol_L (P) \quad \forall P \in \operatorname{Pol}(V)$.
    \end{itemize}
    In other word, $\MVol_L$ is the polarization of $\Vol_L: \operatorname{Pol}(L_\R) \to \R_+$.
\end{definition}

\begin{claim}
    $\MVol_L(P_1, \dots, P_n) = \frac{1}{n!} \sum_{l = 1}^n (-1)^{n - l} \sum_{1 \le i_1 \le \dots \le i_l \le n} \Vol_L (P_{i_1} + \dots + P_{i_l})$.
\end{claim}
\begin{proof}
    Cf. \cite[Thm 3.7, p.118]{Ewald1996}.
\end{proof}

\begin{remark}
    For any subsets $S_1, \dots, S_n \subset L$ we have that $\MVol(\Conv S_1, \dots, \Conv S_n)$ is an integer.
\end{remark}

\begin{theorem}[Kouchnirenko-Bernstein]
    Let $A_1, \dots, A_n \subset M$ be finite subsets of the character lattice and $\Delta_i \defeq \Conv_{M_\R} A_i$ be the corresponding Newton Polytopes. Let $k = \bar k$. Then for the general $\mathbf f \in k^{A_\bullet}$ the system $f_1 = \dots = f_n = 0$ has $\MVol_M(\Delta_1, \dots, \Delta_n)$ solutions in $T^n$.
\end{theorem}
\begin{proof}
    See \cite{bernstein} for the case $k = \C$ and for the arbitrary field see \cite{Kushnirenko1977} --- the author wrote the proof only for $k = \C$ but since it is purely algebraic the proof is valid over arbitrary algebraically closed field. In fact the proof in \cite{bernstein} also does not rely on any techinques that work exclusively in zero characteristic so it may be adapted to work in a purely algebraic setting as well.
\end{proof}

\subsection{Technical Toolkit}\label{ssec:prelim_technical}
\subsubsection{(Geometric) Irreducibility}
Studying the Newton Polytopes theory over the fields that are not algebraically closed could seem odd. For example, we could take $A = \{1 , x, x^2\}$ and it is well-known
%to most high school students 
that the subspace of polynomials of $\R^A$ that have the same number of roots is bounded by a paraboloid, in particular, it is neither open nor closed in the Zariski topology. The same sort of thing happens with the irreducibility. However, the following more stable notion comes in useful:

\begin{definition}
    The $k$-scheme $W$ is called \textbf{geometrically irreducible} if for any field extension $L/k$ the base change $L$-scheme $W_L = W \times_k L$ is irreducible.
\end{definition}

\begin{remark}
    A $k$-scheme of finite type is geometrically irreducible if and only if its base change with respect to $\bar k$ (algebraic closure) is irreducible.
\end{remark}

\begin{remark}
    We are interested in the irreducibility of the general fibre of the projection $X_\P \to \P$. We could also study the irreducibility of the generic fibre, i.e. the fibre of the generic point of the scheme $\P$. Generally speaking, the irreducibility of the generic fibre must not imply the irreducibility of the general fibre. However, it is the case if we work with the geometric irreducibility (and some mild assumptions on the morphism) as follows from \cite[IV.3, 9.7.8]{EGA}. 
\end{remark}

The following theorem gives a usable form to the above speculation.

\begin{theorem}\label{thm:fibre_condition}
    Let $W \to Y$ be a dominant finite type mophism of noetherian schemes such that $Y$ is irreducible and the fibred square $W \times_Y W$ is irreducible. Then the general fibre of $W \to Y$ is geometrically irreducible, i.e. there is a non-empty open subset $U \subset Y$ such that for any $y \in U$ the fibre $W_y$ is geometrically irreducible.
\end{theorem}
\begin{proof}
We will assume without loss of generality that $Y$ is affine. Let $\eta \in Y$ be the generic point of $Y$. By \cite[IV.3, 9.7.8]{EGA}, we only need to show that the generic fibre $W_\eta$ is geometrically irreducible.

If $W$ is not irreducible, then let $W = W_1 \cup \dots \cup W_r$ be its decomposition into irreducible components. Then we have the decomposition into distinct closed subsets: $W \times_Y W = \bigcup_{1 \le i, j \le r} W_i \times_Y W_j$, so $W \times_Y W$ is not irreducible and the statement of the theorem is trivially satisfied. From now on we assume that $W$ is irreducible.

In fact, we could also assume that $W$ is affine. Indeed, let $W = U_1 \cup \dots \cup U_n$ be an open covering such that $U_i$ are affine and non-empty. Since the generic point of $W$ lies in each $U_i$, we get that the fibres $U_{i\eta}$ give an affine open covering of $W_\eta$ such that for all indicies $i, j$ the intersection $U_{i\eta} \cap U_{j\eta}$ is non-empty. Then if we prove that all $U_{i\eta}$ are geometrically irreducible, we will also get that $W_\eta$ is geometrically irreducible. So, we assume without loss of generality that $W$ is affine.

Since $W \times_Y W$ is irreducible, we get that $(W \times_Y W)_\eta = W_\eta \times_\eta W_\eta$ is also irreducible. Now we are left with an algebraic statement to prove: if $A$ is a $K$-algebra ($K \defeq k(\eta)$) with no zero divisors and $A \otimes_K A$ has no zero divisors except nilpotents, then $\Spec A$ is geometrically irreducible. First, note that we can replace $A$ with $A_{red}$, so without loss of generality $A$ is integral. Now, consider $Q$ --- the fraction field of $A$. By \citestacks{037N} we can just show that $Q$ is geometrically irreducible over $K$. By \citestacks{0G33} it is sufficient to prove that $K$ is separably closed in $Q$. Assume the contrary: there is $\alpha \in Q$ that is separably algebraic over $K$ and $\alpha \not \in K$. Then $K(\alpha) \otimes_K K(\alpha)$ contains non-nilpotent (because $\alpha$ is separable) zero divisors. Localization cannot add non-nilpotent zero divisors if there were none, so $Q \otimes_K Q$ must have no non-nilpotent zero divisors. Since $K(\alpha) \otimes_K K(\alpha)$ is a subalgebra of $Q \otimes_K Q$, we get that $Q \otimes_K Q$ also has non-nilpotent zero divisors, which is a contradiction. Hence, $K$ is separably closed in $Q$ and $\Spec A$ is geometrically irreducible over $K$.
\end{proof}

\begin{corollary}\label{cor:fibre_criterion}
    Let $W \to Y$ be a flat dominant finite type mophism of noetherian schemes and $Y$ be irreducible. Then $W \times_Y W$ is irreducible if and only if the general fibre of $W \to Y$ is geometrically irreducible.
\end{corollary}
\begin{proof}
    In one direction we are done by the above theorem. So, assume that $W \times_Y W$ is not irreducible. Again, by \cite[IV.3, 9.7.8]{EGA} we need to show the the generic fibre $W_\eta$ is not geometrically irreducible with $\eta$ being the generic point of $Y$. Assume the contrary: $W \times_Y W$ is not irreducible and $W_\eta$ is geometrically irreducible. Since $W \to Y$ is flat, so is the fibred square $W \times_Y W \to W$. The morphism $\eta \to Y$ is dominant, so $W_\eta \times_\eta W_\eta \to W \times_Y W$ being a flat base change of a dominant morphism also must be dominant. Therefore, $W_\eta \times_\eta W_\eta$ cannot be irreducible. Now, let $\xi \in W_\eta$ be the generic point. $\xi \to W_\eta$ is dominant, so $W_\eta \times_\eta k(\xi) \to W_\eta \times_\eta W_\eta$ is dominant, hence $W_\eta \times_\eta k(\xi)$ is not irreducible and $W_\eta$ is not geometrically irreducible --- contradiction.
\end{proof}

When using the above theorem the following lemmas comes in useful:
\begin{lemma}
    Let $W$ be a Jacobson scheme, $Z \subset X$ be a subset with the induced subspace topology. Assume that for any point $p \in Z$ that is closed in $W$ we have $\dim_p Z < \dim_p W$. Then $W \backslash Z$ is dense in $W$.
\end{lemma}
\begin{proof}
    Assume the contrary: there is an open subset $U \subset X$ such that $U \cap (W \backslash Z) = ~\varnothing$, i.e. $U \subset Z$. Since $W$ is Jacobson, there is a point $p \in U$ that is closed in $W$. We have $\dim_p W = \dim_p U = \dim_p Z$, which is a contradiction.
\end{proof}

\begin{corollary}[Irrelevant fibres]\label{cor:irrelevant_fibres}
    Let  $W \to Y$ be a morphism of $k$-schemes locally of finite type. Let $Z \subset Y$ be a locally closed subscheme such that for all closed $p \in Z$ and all closed $x \in W_p$ we have
    \[
    \dim_x W_p < \dim_x W - \dim_p Z 
    \]
    Then\footnote{by $W_Z$ we denote the pre-image of $Z$ under the morphism $W\to Y$} $W \backslash W_Z$ is dense in $W$. 
\end{corollary}
\begin{proof}
    The question is local on $Y$, so we may assume that $Z$ is closed in $Y$. Now, $W_Z$ is a closed subscheme of $W$, hence the points of $W_Z$ that are closed in $W$ are just the closed points of $W_Z$. Now, we have:
    \[
    \dim_x W_p = \dim_x (W_Z)_p \ge \dim_x W_Z - \dim_p Z.
    \]
     Regrouping the terms and applying the inequality from the assumption we get
    \[
    \dim_x W_Z \le \dim_x W_p + \dim_p Z < \dim_x W.
    \]
    As schemes locally of finite type over a field are Jacobson, we are done.
\end{proof}
\begin{notation}
    For the purposes of this paper we assume that $\dim \varnothing = -\infty$.
\end{notation}
\begin{corollary}[Irreducibility Criterion]\label{cor:irreducibility_criterion}
    Let $W \to Y$ be a dominant morphism of $k$-schemes locally of finite type with equidimensional fibres\footnote{i.e. for any fibre all irreducible components have the same dimension}. Let $Y$ be irreducible with the generic point $\eta$. Assume that the following subsets are locally closed:
    \[
    Y_r \defeq \{ p \in Y\ |\ \dim W_p = \dim W_\eta + r \}
    \]
    Then $W$ is irreducible if and only if all of the following conditions are satisfied:
    \begin{enumerate}
        \item For all closed $p \in W$ we have $\dim_p W \ge \dim W_\eta + \dim Y$;
        \item $W_{Y_0}$ is irreducible;
        \item For all $r > 0$ we have $\dim Y > \dim Y_r + r$.
    \end{enumerate}
\end{corollary}
\begin{proof}
    Assume that $W$ is irreducible. Then since $W \to Y$ is dominant, we have that $\dim W_\eta = \dim W - \dim Y$, which gives us condition 1, because irreducible schemes are equidimensional. By the Chevalley upper semi-continuity theorem $W_{Y_0}$ is an open subset of $W$, hence $W_{Y_0}$ is irreducible. Finally, if there is $r > 0$ such that $\dim Y \le \dim Y_r + r$, then we have 
    \[
    \dim W_{Y_r} = \dim Y_r + \dim W_\eta + r \ge \dim W.
    \]
    By the same Chevalley theorem $W_{Y_r}$ is a locally closed subset of $W$. Since $W_{Y_r}$ is of dimension at least $\dim W$ and $W$ is irreducible, we have that $W_{Y_r}$ is dense in $W$. However, $W_{Y_0}$ is a non-empty open subset of $W$ that does not intersect $W_{Y_r}$, which gives us a contradiction.

    Now, assume all the conditions are satisfied. Take any closed $y \in Y_r$ and any closed $p \in W_y$. We have that
    \[
    \dim_p W \ge \dim W_\eta + \dim Y > \dim W_\eta + r + \dim Y_r = \dim W_y + \dim Y_r,
    \]
    so by Irrelevant Fibres \autoref{cor:irrelevant_fibres} $W\backslash W_Z = W_{Y_0}$ is dense in $W$. Since $W_{Y_0}$ is irreducible, so is $W$.
\end{proof}

\subsubsection{Codimension of quasi-subtori}\label{sssec:techincal_codim}
This subsubsection is purely technical: we just prove the claim \ref{claim:prelim_quasisubtori}. The reader may skip this subsubsection until they come across a reference to the said claim in one of the proofs.

\begin{definition}
    Let $L$ be a lattice and $B \subset L$ be a subset. Then by $\dim B$ we denote the rank of the minimal lattice that contains the set $B - B \defeq \{ b - b'\ |\ b, b' \in B \}$.
\end{definition}

\begin{remark}
    $\dim B = \dim \Conv_{L_\R} B$ for any finite $B \subset L$.
\end{remark}

\begin{remark}
     $\sum (B_i - B_i) = \sum B_i - \sum B_i$ for any collection of subsets $B_1, \dots, B_r \subset ~L$.
\end{remark}

\begin{remark}
    Bellow we will be working with the character lattice $M$. The multiplicative notation is more common for characters, so when we write $B - B$ for $B \subset M$ we mean the set $\{ \chi_1 \cdot \chi_2^{-1}\ |\ \chi_1, \chi_2  \in B \}$.
\end{remark}

\begin{lemma}
    Let $B \subset M$ be a subset such that $1 \in B$. Then the subvariety 
    \[
    V \defeq \{ p \in T^n\ | \chi(p) = 1 \ \forall \chi \in B \}
    \]
    is a quasi-torus of codimension $\dim B$ in $T^n$.
\end{lemma}
\begin{proof}
    Let $H$ be a the sublattice\footnote{$H$ is also the minimal sublattice containing $B - B$ because $1 \in B$. } of $M$ generated by $B$. Clearly $\forall p \in V, \chi \in H$ we have that $\chi(p) = 1$, so we can replace $B$ with $H$. By the Smith Normal Form theorem there is a change of coordinates such that in terms of the new coordinates $x_1, \dots, x_n$ on $T^n$ the lattice $H$ is generated by $x_1^{s_1}, \dots, x_r^{s_r}$, $r = \rk H$, $s_i > 0$. So, $V$ is defined by equations $x_1^{s_1} = \dots = x_r^{s_r} = 1$ --- clearly it is a quasi-torus of codimension $r$.
\end{proof}

\begin{claim}\label{claim:prelim_quasisubtori}
    Let $B_1, \dots, B_r \subset M$ be non-empty subsets. Then the subvariety 
    \[
    V \defeq \{ (p, q) \in T^n \times T^n\ |\ (\chi_2 \cdot \chi_1^{-1}) (p) = (\chi_2 \cdot \chi_1^{-1})(q)\  \forall \chi_1, \chi_2 \in B_i\ \forall i\} 
    \]
    is a quasitorus in $T^n \times T^n$ of codimension $\dim \sum B_i$.
\end{claim}
\begin{proof}
    Consider the antidiagonal embedding $\alpha: M \to M^2, \chi \mapsto (\chi, \chi^{-1})$. Then $V$ is defined by equations $\alpha(\chi)(p, q) = 1 \ \forall \chi \in B_i - B_i\ \forall i$. Clearly $1 \in B_i - B_i$, so $V$ is defined by the equations $\alpha(\chi)(p, q) = 1\ \forall \chi \in \sum (B_i - B_i)$. By the above lemma $V$ is a quasi-torus of codimension $\dim \alpha\left(\sum (B_i - B_i) \right)$. Since $\alpha$ is an embedding, it preserves the dimensions, so
    \[
    \codim V = \dim \sum (B_i - B_i) = \dim \sum B_i.
    \]
\end{proof}

\section{General Sufficient Condition}\label{sec:general}
Recall from \autoref{def:evaluation} that $\E|_\P: \P \times T^n \to \A^m_k \times T^n$ is the evaluation morphism that sends $(\mathbf f, p) \mapsto (\mathbf f(p), p)$ and that $X_\P = \Ker \E|_\P$. Denote by $\E|^2_\P$ the morphism $\P \times T^n \times T^n \to \A^{2m}_k \times T^n \times T^n$, $(\mathbf f, p, q) \mapsto (\mathbf f(p), \mathbf f(q), p, q)$ --- just like $\E|_\P$ it is a vector bundle morphism, i.e. it is fiberwise linear.

\begin{theorem}\label{thm:sufficient}
    Consider the locally closed subsets 
    \[
    \mathcal S_r \defeq \{ (p, q) \in T^n \times T^n\ |\ \rk_{(p, q)} \E|_\P^2 = 2m - r \}.
    \]
    If $\dim \mathcal S_r + r < 2n$ for all $r > 0$, then the general fibre of $X_\P \to \P$ is geometrically irreducible, i.e. the general system from $\P$ defines a geometrically irreducible variety in $T^n$.
\end{theorem}
\begin{proof}
    If $X_\P \to \P$ is not dominant, then the general fibre is empty, in particular it is geometrically irreducible, so from now on we assume that $X_\P \to \P$ is dominant.
    We will show that $X_\P \times_\P X_\P$ is irreducible using corollary \ref{cor:irreducibility_criterion}, which by \autoref{thm:fibre_condition} will imply that the general fibre of $X_\P \to \P$ is geometrically irreducible. Clearly $X_\P \times_\P X_\P = \Ker \E|_\P^2$, so all the fibres of $X_\P \times_\P X_\P \to T^n \times T^n$ are vector spaces, in particular the fibres are equidimensional. 
    
    Since we are interested in the geometric irreducibility, we assume without loss of generality that $k = \bar k$. Take any point $(\mathbf f, p, q)$ from $\P \times T^n \times T^n(k)$ s.t. $\mathbf f(p) = \mathbf f(q) = 0$, i.e. any closed point from $X_\P \times_\P X_\P$. The condition $\E|_\P^2 = 0$ gives no more than $2m$ independent equations near $(\mathbf f, p, q)$, so 
    \[
    \dim_{(\mathbf f, p, q)} X_\P \times_\P X_\P \ge \dim \P \times T^n \times T^n - 2m = \dim \P + 2n - 2m.
    \]
    Clearly $\mathcal S_0$ is open. Since $T^n = \bigsqcup_{r > 0} \mathcal S_r$ and $\dim \mathcal S_r < \dim T^n \times T^n$ for all $r > 0$, we have that $\mathcal S_0 \ne \varnothing$, in particular $\dim \P - 2m$ is the dimension of the generic fibre and $\mathcal S_r$ are precisely the subschemes where the fibre dimension jumps by $r$. Finally, for all $r \ge 0$ the variety $(X_\P \times_\P X_\P)_{\mathcal S_r}$ is a vector bundle over $\mathcal S_r$,  in particular $(X_\P \times_\P X_\P)_{\mathcal S_0}$ is irreducible. By corollary \ref{cor:irreducibility_criterion} $X_\P \times_\P X_\P$ is irreducible.
\end{proof}

\section{Khovanskii Theorems}\label{sec:khovanskii}

The following theorems were proved by Askold Khovanskii in \cite{Khovanskii2016} for $k = \C$. We generalize them to arbitrary field. As before, $M$ is the character (monomial) lattice and for finite subsets $A \subset M$ we define $k^A \defeq \{ f \in \Gamma(T^n_k, \mathcal O)\ | \Supp f \subset A \}$.

\begin{definition}\label{def:khovanskii_condition}
    We say that a collection of subsets $\Delta_1, \dots, \Delta_m \subset M$ satisfy the \textbf{Khovanskii condition} if for any non-empty subset $J \subset \{ 1, \dots, m \}$ we have that\footnote{see \autoref{sssec:techincal_codim} for the definition of dimension} $\dim \sum_{j \in J} \Delta_j > |J|$.
\end{definition}

\begin{theorem}[Irreducibility]\label{thm:khovanskii_irreducibility}
    Let $A_1, \dots, A_m \subset M$ be finite subsets. If $A_1, \dots, A_m$ satisfy the Khovasnkii condition, then the general system from $k^{A_\bullet} = k^{A_1} \times \dots \times k^{A_m}$ defines a geometrically irreducible variety in $T^n$.
\end{theorem}

\begin{definition}
    For a collection $A_1, \dots, A_m$ and a non-empty indices subset $J \subset ~\{ 1, \dots, m \}$ we define \textbf{the defect of $J$}: $\delta(J) \defeq \dim \left( \sum_{j \in J} A_j \right) - |J|$.
\end{definition}

\begin{theorem}[Irreducible Components]\label{thm:khovanskii_components}
    We use the same notation as in the above theorem and denote by $\Delta_i$ the convex hulls of $A_i$ in $M_\R$. Let $N$ be the number of geometric irreducible components of the variety defined by the general system from $k^{A_\bullet}$ in $T^n$. Then we have the following alternative\footnote{note that for a given indices subset one can compute the defect in two ways: for $A_1, \dots, A_m$ (using dimension introduced in \autoref{sssec:techincal_codim}) and $\Delta_1, \dots, \Delta_m$ (using usual polytope dimension) and they will always coincide}:
    \begin{enumerate}
        \item If $\delta(J) > 0$ for all non-empty $J \subset \{1, \dots, m \}$, then $N = 1$.

        \item If there is a non-empty subset  $J \subset \{1, \dots, m \}$ such that $\delta(J) < 0$, then $N = 0$, meaning that the general system from $k^{A_\bullet}$ has no solutions.

        \item If $\delta(J) \ge 0$ for all $J \subset \{ 1, \dots, m \}$ and for some non-empty subset the defect is zero, then there is the greatest subset $J_0$ such that $\delta(J_0) = 0$ and $N =~\MVol_L (\Delta_j)_{j \in J_0}$, where $L$ is the minimal saturated sublattice of $M$ such that\footnote{i.e. $\chi_i \chi_j^{-1} \in L \ \forall \chi_i \in A_i$, $\chi_j \in A_j$ for any two $i, j \in J_0$.} $A_i - A_j \subset L$ for all $i, j \in J_0$
    \end{enumerate}
\end{theorem}

\paragraph{Proof of the Irreducibility theorem}
We will use \autoref{thm:sufficient}. Without loss of generality\footnote{because the characters are invertible over $T^n$, so we can multiply any equation by any character, i.e. shift $A_i$.} $1 \in A_i$ for all $1 \le i \le m$. Then $A_i = \{ 1, \chi^i_1, \dots, \chi^i_{r_i} \}$ and over $(p, q) \in T^n \times T^n$ we have:
\[
\E^2(p, q) =  
    \begin{pmatrix}
        1 & \chi_1^1(p) & \cdots & \chi_{r_1}^1(p) & 0 & 0 & \cdots & 0 & \cdots & 0 & 0 & \cdots & 0 \\
        0 & 0 & \cdots & 0 & 1 &  \chi^2_1(p) & \cdots & \chi^2_{r_2}(p) & \cdots & 0 & 0 & \cdots & 0 \\
        \vdots & \vdots & \ddots & \vdots & \vdots & \vdots & \ddots & \vdots & \ddots & \vdots & \vdots & \ddots & \vdots \\
        0 & 0 & \cdots & 0 & 0 & 0 & \cdots & 0 & \cdots & 1 & \chi^m_1(p) & \cdots & \chi^m_{r_m}(p) \\
        1 & \chi_1^1(q) & \cdots & \chi_{r_1}^1(q) & 0 & 0 & \cdots & 0 & \cdots & 0 & 0 & \cdots & 0 \\
        0 & 0 & \cdots & 0 & 1 &  \chi^2_1(q) & \cdots & \chi^2_{r_2}(q) & \cdots & 0 & 0 & \cdots & 0 \\
        \vdots & \vdots & \ddots & \vdots & \vdots & \vdots & \ddots & \vdots & \ddots & \vdots & \vdots & \ddots & \vdots \\
        0 & 0 & \cdots & 0 & 0 & 0 & \cdots & 0 & \cdots & 1 & \chi^m_1(q) & \cdots & \chi^m_{r_m}(q)
    \end{pmatrix}
\]
Clearly the $i$-th and the $(i + m)$-th rows of this gigantic matrix are proportional if and only if they coincide. Now, for subset $J \subset \{ 1, \dots, m \}$ define the locally closed subschemes\footnote{note that in particular if $\chi(p) = \chi(q)\ \forall \chi \in A_i$ and $i \not \in J$, then $(p, q) \not \in \mathcal S_J$}
\[
\mathcal S_J \defeq \left\{ (p, q) \in T^n \times T^n\ |\ \Big(j \in J \iff \chi(p) = \chi(q)\Big)\forall \chi \in A_j \right\}
\]
Recall that $\mathcal S_r$ is the subset of points of $T^n \times T^n$ where $\rk \E^2$ is $2m - r$. Then we have the decomposition $\mathcal S_r = \bigcup_{|J| = r} \mathcal S_J$. Now we see that the condition that $\dim \mathcal S_r +~r <~2n$ is equivalent to $\codim \mathcal S_J - |J| > 0$. By claim \ref{claim:prelim_quasisubtori} $\codim \mathcal S_J \ge \dim \sum_{j \in J} \Delta_j$.
\qed

\paragraph{Proof of the Irreducible Components Theorem}
The first case is clear. 

The second case. Without loss of generality $\delta(\{1, \dots, s \}) < 0$. After shifting $\Delta_i$ and choosing appropriate coordinates we get that $A_1, \dots, A_s$ are contained in the sublattice generated by $x_1, \dots, x_l$, $l = \dim \sum_j \Delta_j$, in particular $l < s$. If the coefficients are generic enough, then subsystem of first $l$ equations has only finitely many solution, so if we add any non-trivial equation with coefficients generic enough, there will be no solutions. More formally: we can view the subsystem of the first $s$ equations as a square system. Mixed volume of polytopes that are contained in one hyperplane is equal to zero, so by the Bernstein-Kouchnirenko formula the first $s$ equations of the general system have no common solution, so the whole system has no solutions. 

The third case. Note that for any two subsets $J, J' \subset \{1, \dots, m \}$ we have that\footnote{$\dim \sum_{j \in J \cup J'} \Delta_j - |J \cup J'| \le \dim \sum_{j \in J} \Delta_j + \dim \sum_{j \in J'} \Delta_j - \dim \sum_{j \in J' \cap J} \Delta_j - |J \cup J'| = \delta(J) + \delta(J') - \delta(J' \cap J)$} $\delta(J \cup J') \le \delta(J) + \delta(J') - \delta(J \cap J') \le \delta(J) + \delta(J')$ as $\delta \ge 0$. Then $J_0 = \bigcup_{\delta(J) = 0} J$. Wihtout loss of generality $J_0 = \{ 1, \dots, s\}$. After choosing appropriate coordinates $x_1, \dots, x_n$ in torus $T^n$ and shifting $\Delta_i$ we may assume that $\Delta_1, \dots, \Delta_s$ are contained in the sublattice of $M$ generated by $x_1, \dots, x_s$. The idea is as follows. The subsystem of the first $s$ equations must define a finite number (determined by Bernstein-Kouchnirenko) of shifted subtori and each shifted subtorus will contain a single irreducible component. Now denote by $k^{A_{\le q}} \defeq k^{A_1} \times ~\cdots ~\times ~k^{A_q}$, then we have $X(A_{\le s}) \defeq \{ (\mathbf f, p) \in k^{A_{\le s}} \times T^s\ |\ \mathbf{f}(p) = 0 \}$ and the commutative diagram\footnote{everything is well-defined because we shifted $\A_i$ so that $f_i$ depend only on $x_1, \dots, x_s$ for $i < s$.}:
\[
\begin{tikzcd}
    T^n \arrow{r} & T^s \\
    \arrow{u} X(A_\bullet) \arrow{r} \arrow{d} & X(A_{\le s}) \arrow{u} \arrow{d} \\
    k^{A_\bullet} \arrow{r} & k^{A_{\le s}} 
\end{tikzcd}
\]
After all the shifts and the changes of coordinates we have that\footnote{recall that we defined $L$ as the minimal saturated sublattice of $M$ such that $\chi_i \chi_j^{-1} \in L\ \forall \chi \in A_i$, $\chi_j \in A_j\ i, j \in J_0$ } $L$ is generated by $x_1, \dots, x_s$. Let us denote by $A_i/L$ the images of $A_i$ under the projection $M \to M/L$. Then we have the following projection, which is just the evaluation of $x_1, \dots, x_s$ on the given solution of the first $s$ equations:
\[
k^{A_\bullet} \times_{k^{A_{\le s}}} X(A_{\le s}) \to k^{A_{> s}/L}
\]
This morphism in turn gives us the commutative diagram
\[
\begin{tikzcd}
    X(A_\bullet) \arrow{r} \arrow{d} & X(A_{> s}/L) \arrow{d} \\
    k^{A_\bullet} \times_{k^{A_{\le s}}} X(A_{\le s}) \arrow{r} & k^{A_{> s}/L}
\end{tikzcd},
\]
where $X(A_{>s}/L) \defeq \{ (\mathbf f_{> s}, p) \in k^{A_{> s}/L} \times T^{n - s}\ |\ \mathbf f(p) = 0 \}$ and $T^{n - s}$ is the torus corresponding to the lattice $M/L$. Fix a point $(\mathbf f, p_{\le s}) \in k^{A_\bullet} \times_{k^{A_{\le s}}} X(A_{\le s})$ and denote by $\mathbf f_{> s} \in k^{A_{> s}/L}$ its image. Then one can easily see that the induced morphism of fibres $X(A_\bullet)_{(\mathbf f, p_{\le s})} \to X(A_{> s}/L)_{\mathbf f_{> s}}$ is an isomorphism. For any non-empty subset $J \subset \{ s + 1, \dots, m \}$ we have that 
\begin{multline*}
\dim \sum_{j \in J} \Delta_j/L - |J| = \dim \sum_{j \in J \cup J_0} \Delta_j/L - |J| \ge \dim \sum_{j \in J \cup J_0} \Delta_j - |J| - \dim L
= \\
= \dim \sum_{j \in J \cup J_0} \Delta_j - |J| - |J_0| = \delta(J \cup J_0) > 0,
\end{multline*}
because any subset properly containing $J_0$ must be of positive defect. So the general fibre of $X(A_{> s}/L) \to k^{A_{> s}/L}$ is geometrically irreducible. $X(A_{\le s})$ is irreducible and $k^{A_\bullet}$ is a trivial bundle over $k^{A_{\le s}}$, so $ k^{A_\bullet} \times_{k^{A_{\le s}}} X(A_{\le s})$ is irreducible. The generic fibre of $X(A_{> s}/L) \to k^{A_{> s}/L}$ is the same as the generic fibre of $X(A_\bullet) \to k^{A_\bullet} \times_{k^{A_{\le s}}} X(A_{\le s})$, so the general fibre\footnote{'generic' and 'general' are interchangable by \cite[IV.3, 9.7.8]{EGA}} of the latter morphism must be geometrically irreducible. Finally, we have the factorization $X(A_\bullet) \to k^{A_\bullet} \times_{k^{A_{\le s}}} X(A_{\le s}) \to k^{A_\bullet}$. By the Bernstein-Kouchnirenko Formula, $X(A_{\le s}) \to k^{A_{\le s}}$ is a generically finite morphism of degree $\MVol_L(\Delta_1, \dots, \Delta_s)$, so $ k^{A_\bullet} \times_{k^{A_{\le s}}} X(A_{\le s}) \to k^{A_\bullet}$ must be a generically finite morphism of the same degree and the composition $X(A_\bullet) \to k^{A_\bullet}$ has $\MVol_L(\Delta_1, \dots, \Delta_s)$ geometric irreducible components in its general fibre.
\qed

\section{Engineered Complete Intersections}\label{sec:ECI}
In this final section we formulate and prove the most concrete version of our condition, \autoref{thm:ECI}, it is given in \autoref{ssec:ECI_condition}. In \autoref{ssec:ECI_simpler} we give two simpler versions of \autoref{thm:ECI}. In \autoref{ssec:ECI_critical} we give an application of our method by studying some classes of critical loci and Thom-Bordmann strata. Section \ref{ssec:ECI_other} contains two important remarks on how one could use \autoref{thm:ECI}. Finally, in \autoref{ssec:ECI_proof} we prove \autoref{thm:ECI}.

\subsection{Irreducibility Condition}\label{ssec:ECI_condition}
We use the notation from \autoref{sssec:prelim_engineered}, so we have $c^i_1, \dots, c^i_{d_i} \in A_i,\ 1 \le i \le m$ and we study the irreducibility of the variety defined in $T^n$ by the system $c^i_j * f_i = 0$ for the general $f_1 \in k^{A_1}, \dots, f_m \in k^{A_m}$.

\begin{notation}
    For $c \in \Gamma(T^n, \mathcal O)$ and $\chi \in M$ we denote by $c[\chi] \in k$ the unique element of the field such that $c * \chi = c[\chi] \cdot \chi$, i.e. $c[\chi]$ is the coefficient of $c$ with respect to $\chi$.
\end{notation}

\begin{theorem}\label{thm:ECI}
    Let $A \subset M$ be a finite subset, $c_1, \dots, c_d \in k^A$ be polynomials. Consider the map $\mathbf c: A \to k^d$, $\chi \mapsto (c_1[\chi], \dots, c_d [\chi])$. For a complete flag
    \[
    \mathcal V = ~(0 = ~V_0 \subsetneq ~V_1 \subsetneq ~\dots \subsetneq ~V_d = ~k^d)
    \]
    in $k^d$ we define the sets\footnote{in particular, $\Delta_1 (\mathcal V) = \{ \chi \in A\ |\ \mathbf c(\chi) \in V_1\text{ and } \mathbf c(\chi) \ne 0\}$ } $\Delta_i (\mathcal V) \defeq \mathbf c^{-1}(V_i \backslash V_{i - 1})$, $1 \le i \le d$. 

    If there is a complete flag $\mathcal V$ in $k^d$ such that the sets $\Delta_1(\mathcal V), \dots, \Delta_d(\mathcal V)$ satisfy the Khovanskii condition\footnote{\autoref{def:khovanskii_condition}}, then for the general polynomial $f \in k^A$ the variety cut out in $T^n$ by the system $c_1 * f = \cdots = c_d * f = 0$ is geometrically irreducible.
\end{theorem}

\begin{corollary}
    Let $A_1, \dots, A_m \subset M$ be finite subsets, $c^i_1, \dots, c^i_{d_i} \in k^{A_i}$ be polynomials. Consider the maps $\mathbf c^i: A_i \to k^{d_i}$, $\chi \mapsto (c^i_1[\chi], \dots, c^i_{d_i} [\chi])$. For a complete flag $\mathcal V^i =~ (0 = V_0^i \subsetneq V_1^i \subsetneq \dots \subsetneq V^i_{d_i} = k^{d_i})$ in $k^{d_i}$ we define the sets $\Delta_j (\mathcal V^i) \defeq (\mathbf c^i)^{-1}(V^i_j \backslash V^i_{j - 1})$, $1 \le j \le d_i$. 
    
    If there are complete flags $\mathcal V^1, \dots, \mathcal V^m$ such that the sets $\{ \Delta_j(\mathcal V^i) \}^{1 \le i \le m}_{1 \le j \le d_i}$ satisfy the Khovanskii condition, then for the general polynomials $\mathbf f \in k^{A_1} \times \cdots \times k^{A_m}$ the system $c_1^1 * f_1 = \cdots = c^1_{d_1} * f_1 = \cdots = c^m_1 * f_m = \cdots = c^m_{d_m} * f_m = 0$ defines a geometrically irreducible variety in $T^n$.
\end{corollary}
\begin{proof}
    We follow \autoref{rem:prelim_engin_m_ECI_is_ECI}. We can shift $A_i$ by multiplying the equations $c^i_j * f_i$ by characters from $M$ so that $A_1, \dots, A_m$ are disjoint\footnote{we then replace with $A_i$ with $\chi_i \cdot A_i$ and $c_i$ with $\chi_i c_i$}. Put $A \defeq A_1 \sqcup \dots \sqcup A_m$. Denote by $\hat c^i_j$ the images of $c^i_j$ with respect to the natural embeddings $k^{A_i} \xhookrightarrow{} k^A$. Then we can define $\mathbf c: A \to k^{d_1 + \dots + d_m}$, $\chi \mapsto (\hat c^i_j [\chi])^{1 \le i \le m}_{1 \le j \le d_i}$. 
    
    We have that $k^{d_1 + \dots + d_m} = k^{d_1} \oplus \cdots \oplus k^{d_m}$. In particular, the complete flags $\mathcal V^1, \dots, \mathcal V^m$ in $k^{d_1}, \dots, k^{d_m}$ give the flag $\mathcal V \defeq (V_0 \defeq 0 \subsetneq V_1 \subsetneq \dots \subsetneq V_{d_1 + \dots + d_m} )$ in $k^{d_1 + \dots + d_m}$, where $V_{d_1 + \cdots + d_r + i} = V^1_{d_1} \oplus \cdots \oplus V^r_{d_r} \oplus V^{r + 1}_i$ for $1 \le i \le d_{r + 1}$. Moreover, the subsets $\Delta_q (\mathcal V) = \mathbf c^{-1} (V_q \backslash V_{q - 1})$ satisfy the Khovanskii condition if and only if the collection $\{ \Delta_j(\mathcal V^i)\}^{1 \le i \le l}_{1 \le j \le d_i}$ does as 
    \begin{multline*}
    \Delta_{d_1 + \dots + d_r + i}(\mathcal V) = \mathbf c^{-1} (V_{d_1 + \dots + d_r + i}) \backslash \mathbf c^{-1} (V_{d_1 + \dots + d_r + i - 1}) = \\
    = \left((\mathbf c^{r + 1})^{-1} (V^{r + 1}_i) \bigcup_{j = 1}^r (\mathbf c^j)^{-1} (V^j_{d_j}) \right) \backslash \left((\mathbf c^{r + 1})^{-1} (V^{r + 1}_{i - 1}) \bigcup_{j = 1}^r (\mathbf c^j)^{-1} (V^j_{d_j}) \right) = \\
    = (\mathbf c^{r + 1})^{-1}(V^{r + 1}_i) \backslash (\mathbf c^{r + 1})^{-1} (V^{r + 1}_{i - 1}) = \Delta_i(\mathcal V^{r + 1})
    \end{multline*}
    for $1 \le i \le d_{r + 1}$ and $0 \le r < m$.
    
    Therefore we reduced the problem to the case when $m = 1$ and now instead of $A_1, \dots, A_m$ we have one finite set of monomials $A$; instead of polynomials $c^i_j \in k^{A_i}$ we have the polynomials\footnote{$c_{d_1 + \dots + d_r + i} = c^{r + 1}_i$ for $1 \le i \le d_{r + 1}$} $c_1, \dots, c_d \in k^A$, where $d = d_1 + \dots + d_m$; finally, instead of the maps $\mathbf c^1, \dots, \mathbf c^m$ and the complete flags $\mathcal V^1, \dots, \mathcal V^m$ we have one map $\mathbf c: A \to k^d$ and one complete flag $\mathcal V$. Thus, the above theorem applies.
\end{proof}

\subsection{Some Reductions}\label{ssec:ECI_simpler}
In this section we replace the complete flags in the condition \autoref{thm:ECI} with individual fibres and analyse them instead.

\begin{remark}
    Note that if $\Delta_1, \dots, \Delta_d \subset M$ are subsets such that $\dim \Delta_i > d$ for all $i$, then automatically $\Delta_1, \dots, \Delta_d$ satisfy the Khovanskii condition.
\end{remark}

\begin{claim}
    As above, let $A_1, \dots, A_l \subset M$ be finite subsets, $c^i_1, \dots, c^i_{d_i} \in k^{A_i}$ be polynomials. Consider $\pi_i \defeq \mathbb P(\mathbf c^i): A \to \mathbb P(k^{d_i})$, $\chi \mapsto \left[c^i_1[\chi]: \cdots : c^i_{d_i}[\chi]\right]$. If there are points $p^i_1, \dots, p^i_{d_i} \in \mathbb P(k^{d_i})$ in general position such that all the fibres $\{ \pi_i^{-1} (p^i_j) \}^{1 \le i \le m}_{1 \le j \le d_i}$ satisfy the Khovanskii condition, then the system $c^i_j * f = 0$ defines a geometrically irreducible variety for the general $\mathbf f \in k^{A_\bullet}$.
\end{claim}

By \autoref{rem:prelim_engin_m_ECI_is_ECI} it is sufficient to prove the above claim only for $m = 1$, so reformulate the above claim in this case and prove only that.

\begin{claim}\label{claim:ECI_simpler}
    Let $A \subset M$ and $c_1, \dots, c_d \in k^A$ be polynomials. Consider\footnote{the locus where $\pi$ is undefined, i.e. $\chi \in A$ such that $c_i * \chi = 0$ for all $i$ does not affect the system, so we may assume that it is empty} $\pi \defeq \mathbb P(\mathbf c): A \to \mathbb P(k^d)$, $\chi \mapsto \left[ c_1[\chi] : \cdots : c_d[\chi] \right]$. If there are points $p_1, \dots, p_d \in \mathbb P(k^d)$ in general position\footnote{i.e. no $s$ points are contained in a $s - 2$ projective subspace for any $s > 0$.} such that the fibres $\pi^{-1}(p_1), \dots, \pi^{-1}(p_d)$ satisfy the Khovanskii condition, then the system $c_1 * f = \cdots = c_d * f = 0$ defines a geometrically irreducible variety for the general $f \in k^A$.
\end{claim}
\begin{proof}
    Denote by $l_i \subset k^d$ the one-dimensional subspaces corresponding to $p_i \in \mathbb P(k^d)$. Define the subspaces $V_i \defeq l_1 + \cdots + l_d$. Since $p_1, \dots, p_d$ are in general position, $\dim V_i = i$, so $V_1 \subset V_2 \subset \cdots \subset V_d$ is a complete flag. Clearly we have the inclusions $\pi^{-1} (p_i) = ~\mathbf c^{-1} (l_i) \subset \mathbf c^{-1} (V_i \backslash V_{i - 1})$, so $\mathbf c^{-1} (V_i \backslash V_{i - 1})$ satisfy the Khovanskii condition.
\end{proof}

\begin{corollary}\label{cor:ECI_simpler_poly}
    Let $A \subset M$ be a subset, $l : A \to k$ be a function, and $c_i \in k^A$ be polynomials such that $c_i[\chi] = p_i(l(\chi))\ \ \forall \chi \in A$, where $p_i$ is a polynomial of degree $i - 1$. Assume that there are at least $d$ distinct values $v_1, \dots, v_d \in k$ such that the fibres $\phi^{-1}(v_1), \dots, \phi^{-1}(v_d)$ satisfy the Khovanskii condition.
    Then $c_1 * f = \dots = c_d * f = 0$ defines a geometrically irreducible variety in $T^n$ for the general $f \in k^A$.
\end{corollary}
\begin{proof}
    Without loss of generality $A = \Delta_1 \sqcup \Delta_2 \sqcup \cdots \sqcup \Delta_d$. The polynomials $p_1, \dots, p_d$ by definition form a basis of polynomials in one variable of degree $\le d$. Hence, after applying a linear invertible operation on the system $c_1 * f = \cdots = c_d * f = 0$ we could assume that $p_i(x) = x^{i - 1}$. Then $\pi : A \to \mathbb P(k^d)$ takes form $\chi \mapsto [1: \phi(\chi) : \phi(\chi)^2 : \cdots : \phi(\chi)^{d - 1}]$. Since the image of $\pi$ lies on the Veronese curve of degree $d - 1$, any $d$ points from $\pi(\Delta_1 \cup \cdots \cup \Delta_d)$ are in general position. Hence, if there are $d$ distinct fibres that satisfy the Khovanskii condition, then their images are in general position and \autoref{claim:ECI_simpler} applies.
\end{proof}

\begin{remark}
    Everywhere above instead of analysing fibres we could take their subsets that satisfy the Khovanskii condition. It could make great difference e.g. when using \autoref{cor:ECI_simpler_poly}: in fact, we do not need $c_i$ to exhibit polynomial behavior on the whole support set $A$, but only on some sufficiently big subsets of fibres of $\phi$.
\end{remark}

\begin{remark}\label{rem:ECI_simpler_unsimpler}
    Everywhere above we could replace the last point $p_d$ (or the fibre $\phi^{-1}(v_d)$ in \autoref{cor:ECI_simpler_poly}) with the complement of $\{ p_1, \dots, p_{d - 1} \}$ (or the pre-image of the complement of $\{ v_1, \dots, v_{d - 1} \}$) given that the complement is not contained in the span of $p_1, \dots, p_d$ (no additional assumption is needed in \autoref{cor:ECI_simpler_poly}). The proofs will go exactly the same except that in the proof of \autoref{claim:ECI_simpler} we will replace $l_d$ with the whole space $k^d$.
\end{remark}

\subsection{Applying \autoref{thm:ECI}}\label{ssec:ECI_other}
\subsubsection{Verifying combinatorial condition: Exhaustive search}
In this subsection we give an explicit algorithm that allows one to use our sufficient condition \autoref{thm:ECI} to the full extent. This algorithm may be difficult to use manually, but it is perfectly possible to run a computer program for any specific support set $A$ and polynomials $c_1, \dots, c_d \in k^A$ to determine whether the conditions of \autoref{thm:ECI} are satisfied (and in that case for the general $f \in k^A$ the system $c_1 * f = \cdots = c_d * f = 0$ defines a geometrically irreducible variety.).
\begin{enumerate}
    \setcounter{enumi}{-1}
    \item We are given a finite support set $A \subset M$ and polynomials $c_1, \dots, c_d \in k^A$. By \autoref{rem:prelim_engin_m_ECI_is_ECI} our algorithm also applies to the general case with multiple support sets.
    
    \item Choose an order $\prec$ on $A$. Let $A = \{ \chi_1, \dots, \chi_m \}$ such that $\chi_1 \prec \dots \prec \chi_m$.

    \item Define $s_i \defeq \min \{ s\ |\ \dim \langle \mathbf c[\chi_1], \dots, \mathbf c[\chi_s] \rangle = i \}$. Since $c_1, \dots, c_d$ are linearly independent we can find such $s_1, \dots, s_d$. Define the complete flag $\mathcal V$:
    \[
    \mathcal V \defeq (V_1 \subset \cdots \subset V_d), \quad V_i \defeq \langle \mathbf c[\chi_1], \dots, \mathbf c[\chi_{s_i}] \rangle.
    \]
    
    \item If the flag $\mathcal V$ satisfies the conditions of \autoref{thm:ECI} (i.e. $\mathbf c^{-1} (V_i \backslash V_{i - 1})$ satisfy the Khovasnkii condition), then we are done: for the general $f \in k^A$ the system $c_1 * f = \cdots = c * f = 0$ defines a geometrically irreducible variety. If the flag $\mathcal V$ does not satisfy the conditions of \autoref{thm:ECI}, then we choose another order in item 1. and repeat.

    \item If we iterated over all orders on $A$ and never constructed a flag that satisfies the conditions of \autoref{thm:ECI}, then our sufficient condition of irreduciblity is inapplicable.
\end{enumerate}

\subsubsection{Explicit genericity conditions}
We are building on the notion of Engineered Complete Intersection which is developed in \cite{Esterov2024}. In particular, the work contains a sufficient genericity condition for the ECI: for a fixed finite subset $A \subset M$ and polynomials $c_1, \dots, c_d \in \C^A$ for all $f \in \C^A$ such that the system $c_1 * f = \cdots = c_d * f = 0$ is smooth and non-degenerate upon cancellations (\cite[Def. 1.6.2]{Esterov2024}), the varieties defined by these systems are diffeomorphic (\cite[Prop. 4.14.2]{Esterov2024}), a cancellation matrix is given in \cite[Cor. 4.6]{Esterov2024}. As irreducibility is equivalent to connectedness for smooth varieties, we get a sufficient condition of irreducibility for an ECI: if $A, c_1, \dots, c_d$ satisfy the conditions of \autoref{thm:ECI} and $f \in \C^A$ is such that the system $c_1 * f = \cdots = c_d * f = 0$ is a non-degenerate upon cancellations system that defines a smooth variety, then that variety is irreducible.

Such an explicit genericity condition in arbitrary characteristic would require a theory of Engineered Complete Intersections similar to the one developed in \cite{Esterov2024} but over a field of arbitrary characteristic (rather than just $\C$).

\subsection{Critical Loci \& Thom-Brodmann Strata}\label{ssec:ECI_critical}
This section gives a few examples on how one could apply \autoref{thm:ECI} or its corollaries from \autoref{ssec:ECI_simpler}

\begin{notation}
    We denote by $\mathbb F$ the prime subfield of $k$, i.e. $\mathbb F = \Z/p\Z$ if $p = \char k > 0$ and $\mathbb F = \Q$ otherwise.
\end{notation}

\begin{notation}
    Everywhere bellow by degree of a polynomial (with respect to any given coordinate) we will mean not the integer from $\Z$ but its residue class from $\Z/(\char k) \cdot \Z$ (in particular, if $\char k = 0$, then it just the usual integer degree). This way whenever we write $\deg_x f$ we always have that $\deg_x f \in \mathbb F$ and if $\deg_x f = \deg_x g$, then what we really mean is that the actual degrees with respect to $x$ are the same only modulo $\char k$.
\end{notation}

\begin{example}
    Let $A  \subset  \langle x, y_1, \dots, y_n \rangle$ be a finite set. Consider the hyperplane: 
    \[
    H^x_\lambda \defeq \{ \chi \in M\ | \deg_x \chi = \lambda \}.
    \]
    If there is $\lambda \in \mathbb F$ such that $A \cap H^x_\lambda$, $A \backslash H^x_\lambda$ satisfy the Khovanskii condition\footnote{i.e. both of them are at least 2-dimensional and they are not contained in two parallel 2-dimensional planes}, then $f = f'_x = 0$ is geometrically irreducible for the general~$f \in k^A$.
\end{example}
\begin{proof}
    $f = f'_x = 0$ is equivalent to $f = xf'_x = 0$. Consider the polynomials $c_1 = ~\sum_{\chi \in A} ~\chi$ and $c_2 = x c'_{1x} = \sum_{\chi \in A} (\deg_x \chi) \cdot \chi$. We have that $f = c_1 * f$ and $x f'_x = c_2 * f$.
    Now, let $l_\lambda$ be the 1-dimensional subspace of $k^2$ spanned by the vector $(1, \lambda)$. Then $\mathbf c^{-1} (l) = H^x_\lambda$, where $\mathbf c = c_1 \oplus c_2: A \to k^2$. Consider the complete flag $0 \subset l \subset k^2$. By \autoref{thm:ECI} we are done.
\end{proof}

\begin{example}
Let $A \subset \langle x, y, z_1, \dots, y_n \rangle$ be a finite set. Consider the (punctured) planes:
\begin{gather*}
    H^{[x : y]}_0 \defeq \{ \chi \in M\ |\ \deg_x \chi = \deg_y \chi = 0\}; \\
    H^{[x: y]}_{[\lambda : \mu]} \defeq \left\{\chi \in M\ |\ [\deg_x \chi : \deg_y \chi] = [\lambda : \mu ] \right\} \text{ for } [\lambda : \mu] \in \mathbb P(\mathbb F^2).
\end{gather*}
If there is $[\lambda : \mu] \in \mathbb P(\mathbb F^2)$ such that $A \cap H^{[x: y]}_{[\lambda : \mu]}$, $A \backslash (H^{[x:y]}_0 \cup H^{[x: y]}_{[\lambda : \mu]})$ satisfy the Khovanskii condition,
then $f'_x = f'_y = 0$ is geometrically irreducible for the general $f \in k^A$.
\end{example}
\begin{proof}
    Monomials from $H^{[x: y]}_0$ have zero coefficients in both $f'_x, f'_y$, so without loss of generality we may assume that $A \cap H^{[x: y]}_0 = \varnothing$. The system $f'_x = f'_y = 0$ is equivalent to $xf'_x = yf'_y = 0$. Consider the polynomials $c_1 = \sum_{\chi \in A} (\deg_x \chi) \cdot \chi$, $c_2 = \sum_{\chi \in A} (\deg_y \chi) \cdot \chi$.
    Then $xf'_x = c_1 * f$ and $y f'_y = c_2 * f$. Without loss of generality $\deg_x \chi \ne 0$ for all $\chi \in \Delta_1$. Denote by $l \subset k^2$ the one-dimensional subspace of $k^2$ spanned by $(\lambda, \mu)$. Then $\mathbf c^{-1} (l) = H^{[x: y]}_{[\lambda: \mu]}$, where $\mathbf c = c_1 \oplus c_2: A \to k^2$. Consider the complete flag $0 \subset l \subset k^2$. By \autoref{thm:ECI} we are done.
\end{proof}

\begin{example}
    Let $A \subset \langle x, y_1, \dots, y_n \rangle$ be a finite subset. Again, consider the hyperplanes:
    \[
    H^x_\lambda \defeq \{ \chi \in M\ | \deg_x \chi = \lambda \}.
    \]
    If there are $\lambda_1, \dots, \lambda_r \in \mathbb F$ such that $A \cap H^x_{\lambda_1}, \dots, A \cap H^x_{\lambda_r}, A \backslash (H^x_{\lambda_1} \cup \cdots \cup H^x_{\lambda_r})$ satisfy the Khovanskii condition, then the system $f = \frac{\partial}{\partial x} f = \dots = \frac{\partial^r}{\partial x^r} f = 0$ defines a geometrically irreducible variety for the general $f \in k^A$.
\end{example}
\begin{proof}
    First note that over $T^n$ the equation $\frac{\partial^i}{\partial x^i} f = 0$ is equivalent to $x^i \frac{\partial^i}{\partial x^i} f = 0$. Then define $p_i(d) \defeq \frac{d!}{(d - i + 1)!}$ --- clearly these are polynomials in $d$ of degree $i - 1$. Put $c_i \defeq \sum_{\chi \in A} p_i(\deg_x \chi) \cdot \chi$. Then $x^i \frac{\partial^i}{\partial x^i} f = c_i * f$, so we are to study the ECI $c_1 * f = \dots = c_{r + 1} * f = 0$. By \autoref{cor:ECI_simpler_poly} and \autoref{rem:ECI_simpler_unsimpler} this ECI is geometrically irreducible.
\end{proof}

% For a fixed A, c_i the varieties defined by the systems c * f = 0 are diffeomorphic for all f such that the system is regular\footnote{i.e. defines a smooth variety} and NUC upon cancellation with the cancellation matrix defined by Cor. 4.6

% if A, c_i satisfy \autoref{thm:ECI} and f is such that the system c * f = 0 is regular and NUC, then the variety defined by ... is irreducible

\subsection{Proof of \autoref{thm:ECI}}\label{ssec:ECI_proof}

\begin{proof}
    We will use \autoref{thm:sufficient}. In Step 1 we replace the complete flag $\mathcal V$ with the standard complete flag in $k^d$. Then in Step 2 we show that the matrix of  the vector bundle morphism $\mathcal E|^2_\P$ (cf. bellow) is in nice echelon form if $\mathcal V$ is standard and then it is easier to describe the rank drop locus of $\mathcal E|^2_\P$. Finally, in Step 3 we show that the rank drop locus of $\mathcal E|^2_\P$ is small enough for \autoref{thm:sufficient} to hold precisely when $\Delta_1(\mathcal V), \dots, \Delta_d(\mathcal V)$ satisfy the Khovanskii condition.

    \paragraph{Step 1. Without loss of generality $\mathcal V$ is standard.} In $k^d$ there is the standard basis
    \[
    e_1 = 
    \begin{pmatrix}
        1 \\ 0 \\ \vdots \\ 0
    \end{pmatrix},
    e_2 = 
    \begin{pmatrix}
        0 \\ 1 \\ \vdots \\ 0
    \end{pmatrix},
    \dots, 
    e_d = 
    \begin{pmatrix}
        0 \\ 0 \\ \vdots \\ 1
    \end{pmatrix}.
    \]
    Then by the standard complete flag we mean the flag corresponding to the basis $e_1, \dots, e_n$, i.e. $0 \subset \langle e_1 \rangle \subset \langle e_1, e_2 \rangle \subset \dots \subset \langle e_1, \dots, e_d \rangle$, where $\langle - \rangle$ denotes the linear span. The following paragraph shows that we could assume that $\mathcal V$ is standard. 
    
    Take any matrix $G = (\gamma_{ij}) \in \GL_d (k)$ and put $\tilde c_i \defeq \sum_j \gamma_{ij} c_j$. On one hand, the system $\tilde c_1 * f = \cdots = \tilde c_d * f = 0$ is equivalent to $c_1 * f = \cdots = c_d * f = 0$ for all $f \in k^A$. On the other hand, if we define $\tilde{\mathbf c}: A \to k^d$, $\chi \mapsto (\tilde c_1 [\chi], \dots, \tilde c_d[\chi])$ then $\tilde {\mathbf c} = G \circ \mathbf c$, so for any subspace $W \subset V$ we have $\tilde {\mathbf c}^{-1} (W) = \mathbf c^{-1} (G^{-1} (W))$. Since $\GL_d (k)$ acts transitively on complete flags we can choose $G \in \GL_d (k)$ such that $G^{-1}(\mathcal V)$ is the standard complete flag and replace $c_1, \dots, c_d$ with $\tilde c_1, \dots, \tilde c_d$.

    Henceforth we assume that $\mathcal V$ is standard. In particular for all $\chi_i \in \Delta_i$, we have that $c_j * \chi_i = 0$ for $j > i$ (because $\mathbf c(\chi_i) \in \langle e_1, \dots, e_i \rangle$) and $c_i * \chi_i \ne 0$ (otherwise $\mathbf c(\chi_i) \in \langle e_1, \dots, e_{i - 1} \rangle$). We will also assume that $\mathbf c^{-1} (0) = \varnothing$ as monomials from the kernel affect neither the system $c_1 * f_1 = \dots = c_d * f = 0$, nor the sets $\Delta_1, \dots, \Delta_d$.

    \paragraph{Step 2. Analysing the matrix of $\mathcal E|_\P^2$.} Recall that $\E|_\P^2$ is the vector bundle morphism $k^A \times T^n \times T^n \to \mathbb A^{2d} \times T^n \times T^n$ such that
    \[
    (f, p, q) \mapsto ((c_1 * f) (p), (c_1 * f)(q), \dots, (c_d * f)(p), (c_d * f)(q), p, q).
    \]
    So if we put $A = \{ \chi_1, \dots, \chi_N \}$, then the matrix of $\mathcal E|^2_\P$ is:
    \[
    \begin{pmatrix}
        (c_1 * \chi_1) (p) & (c_1 * \chi_2)(p) & \cdots & (c_1 * \chi_N) (p) \\
        (c_1 * \chi_1) (q) & (c_1 * \chi_2)(q) & \cdots & (c_1 * \chi_N) (q) \\
        (c_2 * \chi_1) (p) & (c_2 * \chi_2)(p) & \cdots & (c_2 * \chi_N) (p) \\
        (c_2 * \chi_1) (q) & (c_2 * \chi_2)(q) & \cdots & (c_2 * \chi_N) (q) \\
        \vdots & \vdots & \ddots & \vdots & \vdots \\
        (c_d * \chi_1) (p) & (c_d * \chi_2)(p) & \cdots & (c_d * \chi_N) (p) \\
        (c_d * \chi_1) (q) & (c_d * \chi_2)(q) & \cdots & (c_d * \chi_N) (q) \\
    \end{pmatrix}.
    \]
    If we show that the rank of the above matrix drops by $r$ only at subvarieties of codimension at least $r + 1$ and the general rank is $2d$, then by \autoref{thm:sufficient} our theorem will be proved. Now $\Delta_1, \dots, \Delta_d$ give a partition\footnote{here we use the assumption that $\mathbf c^{-1} (0) = \varnothing$} of $A$. Let us denote $\Delta_i = \{ \chi^i_1, \dots, \chi^i_{N_i} \}$. By the last[CHANGE TO FORMULA REFERENCE] paragraph of the above step we have that $c_i * \chi^j_t = 0$ for $j > i$. Hence, the matrix of $\mathcal E|^2_\P$ takes form:
    \[
    \begin{pmatrix}
        (c_1 * \chi^1_1) (p) & \cdots & (c_1 * \chi^1_{N_1}) (p) & (c_1 * \chi^2_1) (p) & \cdots &  (c_1 * \chi^d_{N_d}) (p) \\
        (c_1 * \chi^1_1) (q) & \cdots & (c_1 * \chi^1_{N_1}) (q) & (c_1 * \chi^2_1) (q) & \cdots &  (c_1 * \chi^d_{N_d}) (q) \\
        0 & \cdots & 0 & (c_2 * \chi^2_1) (p) & \cdots &  (c_2 * \chi^d_{N_d}) (p) \\
        0 & \cdots & 0 & (c_2 * \chi^2_1) (q) & \cdots &  (c_2 * \chi^d_{N_d}) (q) \\
        \vdots & \vdots & \ddots & \vdots & \ddots & \vdots \\
        0 & \cdots & 0 & 0 & \cdots &  (c_d * \chi^d_{N_d}) (p) \\
        0 & \cdots & 0 & 0 & \cdots &  (c_d * \chi^d_{N_d}) (q) \\
    \end{pmatrix}
    \]
    In particular, we have that
    \[
    \rk_{(p, q)} \E|^2_\P \ge \sum_{i = 1}^d \rk
    \begin{pmatrix}
        (c_i * \chi^i_1) (p) & \cdots & (c_i * \chi^i_{N_i}) (p) \\
        (c_i * \chi^i_1) (q) & \cdots & (c_i * \chi^i_{N_i}) (q)
    \end{pmatrix}.
    \]
    Since $c_i * \chi^i_j \ne 0$ (by the last paragraph of the above step[CHANGE TO FORMULA REFRENECE]), $c_i * \chi^i_j$ do not vanish on $T^n \times T^n$, so the above sum is at least $d$ and $\rk \mathcal E \ge d$ everywhere on $T^n \times T^n$.

    \paragraph{Step 3. Codimension of $\mathcal S_r$.} Recall that by $\mathcal S_r$ we denote the following subschemes: 
    \[
    \mathcal S_r = \{ (p, q) \in T^n \times T^n\ |\ \rk_{(p, q)} \E^2|_\P = 2d - r \}.
    \]
    To finish the proof using \autoref{thm:sufficient} we need to show that $\dim \mathcal S_r + r < 2n$ for all $r > 0$. Note that the last sentence of the above paragraph tells us that $\mathcal S_r = \varnothing$ for $r > d$, so $\dim \mathcal S_r = -\infty$ and we only need to tackle the case when $r \le d$. For a subset of characters $\Delta = \{ \chi_1, \dots, \chi_t \} \subset A$ and $p \in T^n$ denote by $\Delta(p)$ the vector $(\chi_1 (p), \dots, \chi_t(p)) \in k(p)^t$. Now, for each non-empty subset $\Delta \subset A$ consider the closed subschemes\footnote{here by $k(p, q)$ we mean the residue field of the point $(p, q) \in T^n$}
    \[
    \mathcal S_\Delta \defeq \{ (p, q) \in T^n \times T^n\ |\ \exists \lambda \in k(p, q)^\times:\  \Delta(p) = \lambda \Delta(q) \},
    \]
    i.e. if $(p, q) \in \mathcal S_\Delta$, then $(\chi_1 \cdot \chi_2^{-1}) (p) = (\chi_1 \cdot \chi_2^{-1}) (q)$ for all $\chi_1, \chi_2 \in \Delta$. Then for a subset of indicies $J \subset \{1, \dots, d\}$ we define\footnote{recall that $\Delta_j = \mathbf c^{-1} (V_j \backslash V_{j - 1})$} $\mathcal S_J \defeq \bigcap_{j \in J} \mathcal S_{\Delta_j}$. Now we will show that $\mathcal S_r \subset \bigcup_{|J| = r} \mathcal S_J$ and that the Khovanskii condition implies $\dim \mathcal S_J + |J| < 2n$. 

    For $(p, q) \in T^n \times T^n$ define the set of indicies 
    \[
    I(p, q) \defeq \{ i \in \{1, \dots, d\}\ |\ (p, q) \not \in \mathcal S_{\Delta_i} \}.
    \]
    If $(p, q) \not \in \mathcal S_J$ for all $|J| = r$, then\footnote{indeed: $(p, q) \in \Delta_j$ for all $j \not \in I(p, q)$, so $(p, q) \in \mathcal S_J$ for $J = \{ 1, \dots, d \} \backslash I(p, q)$ As $|J| < r$, we get $I(p, q) \ge d - r + 1$} $|I(p, q)| \ge d - r + 1$. If $(p, q) \not \in \mathcal S_{\Delta_i}$, then there are $\chi^i_t, \chi^i_s \in \Delta_i$, $t \ne s$, such that $(\chi^i_t(p), \chi^i_s(p))$ is not proportional to $(\chi^i_t(q), \chi^i_s(q))$, i.e.
    \[
    \rk 
    \begin{pmatrix}
        \chi^i_t(p) & \chi^i_s(p) \\
        \chi^i_t(q) & \chi^i_s(q)
    \end{pmatrix} = 2
    \]
    and therefore
    \[
    \rk
    \begin{pmatrix}
        (c_i * \chi^i_t) (p) & (c_i * \chi^i_t) (p) \\
        (c_i * \chi^i_t) (q) & (c_i * \chi^i_s) (q)
    \end{pmatrix}
    =
    \rk \left(
    \begin{pmatrix}
        \chi^i_t(p) & \chi^i_s(p) \\
        \chi^i_s(q) & \chi^i_s(q)
    \end{pmatrix}
    \cdot
    \begin{pmatrix}
        c_i[\chi^i_t] & 0 \\
        0 & c_i[\chi^i_s]
    \end{pmatrix}
    \right) = 
    \rk 
    \begin{pmatrix}
        \chi^i_t(p) & \chi^i_s(p) \\
        \chi^i_t(q) & \chi^i_s(q)
    \end{pmatrix} = 2,
    \]
    the latter equality holds because $c_i[\chi^i_t], c_i[\chi^i_s] \ne 0$. Now recall from the previous paragraph that
    \[
    \rk_{(p, q)} \E^2|_\P \ge \sum_{i = 1}^d \rk
    \begin{pmatrix}
        (c_i * \chi^i_1) (p) & \cdots & (c_i * \chi^i_{N_i}) (p) \\
        (c_i * \chi^i_1) (q) & \cdots & (c_i * \chi^i_{N_i}) (q)
    \end{pmatrix}
    \]
    So we have that\footnote{
    $\begin{pmatrix}
        \chi^i_t(p) & \chi^i_s(p) \\
        \chi^i_t(q) & \chi^i_s(q)
    \end{pmatrix}$ is a submatrix of $\begin{pmatrix}
        (c_i * \chi^i_1) (p) & \cdots & (c_i * \chi^i_{N_i}) (p) \\
        (c_i * \chi^i_1) (q) & \cdots & (c_i * \chi^i_{N_i}) (q)
    \end{pmatrix}$. We also use that that each summand is at least 1 as all the matrices are non-zero
    }
    \begin{multline*}
    \rk_{(p, q)} \E^2|_\P \ge \sum_{i \in I}
    \begin{pmatrix}
        (c_i * \chi^i_1) (p) & \cdots & (c_i * \chi^i_{N_i}) (p) \\
        (c_i * \chi^i_1) (q) & \cdots & (c_i * \chi^i_{N_i}) (q)
    \end{pmatrix} + \\ +
    \sum_{j \in \{ 1, \dots, d \} \backslash I}
    \begin{pmatrix}
        (c_j * \chi^j_1) (p) & \cdots & (c_j * \chi^i_{N_j}) (p) \\
        (c_j * \chi^j_1) (q) & \cdots & (c_j * \chi^j_{N_j}) (q)
    \end{pmatrix} \ge 2 |I| + d - |I| = |I| + d \ge 2d - r + 1 > 2d - r.
    \end{multline*}
    Hence, $(p, q) \not \in \mathcal S_r$, i.e. we showed that $\mathcal S_r$ is contained in $\bigcup_{|J| = r} \mathcal S_J$.

    Finally, we have that $\dim \mathcal S_r \le \dim \bigcup_{|J| = r} \mathcal S_J = \max_{|J| = r} \dim \mathcal S_J$. By \autoref{claim:prelim_quasisubtori} $\mathcal S_J$ has codimension $\dim \sum_{j \in J} \Delta_j$. Therefore by the Khovasnkii condition we have that $\dim \mathcal S_J + |J| = 2n - \dim \sum_{j \in J} \Delta_j + |J| < 2n$ and 
    \[
    \dim \mathcal S_r + r \le \max_{|J| = r} \dim \mathcal S_J + r = \max_{|J| = r} (\dim \mathcal S_J + |J|) < 2n. 
    \]
    So \autoref{thm:sufficient} holds which proves \autoref{thm:ECI}.
\end{proof}

% bibliography
\bibliographystyle{bibstyle}
\bibliography{bibliography} %Import the bibliography file

\end{document}